\documentclass[11 pt]{amsart}
\usepackage{amssymb}
\usepackage{graphicx}
\usepackage{color}
\usepackage{amsmath}
\usepackage{graphicx}
\usepackage{mathrsfs}
\usepackage{amsfonts} 
\usepackage{amscd} 

\newtheorem{theorem}{Theorem}[section]

\newtheorem{definition}[theorem]{Definition}

\newtheorem{remark}[theorem]{Remark}
\newtheorem{proposition}[theorem]{Proposition}

\newtheorem{corollary}[theorem]{Corollary}

\numberwithin{equation}{section}

\renewcommand{\H}{{\mathcal H}}
\def\C{\mathbb C}
\def\R{\mathbb R}

\def\C{\mathbb C}
\def\R{\mathbb R} 

\def\N{\mathbb N}

\def\de{\delta}

\def\et{\eta}
\def\th{\theta}

\def\ve{\varepsilon}
\def\LA{\Lambda}

\def\om{\omega}
\def\va{\varphi}

\def\Ph{\Phi}

\def\g{\mathfrak{g}}

\def\m{\mathfrak{m}}

\def\g{\mathfrak g}

\def\ve{\varepsilon}
\def\si{\sigma}
\def\om{\omega}

\def\ph{\phi}
\def\ch{\chi}

\def\ps{\psi}
\def\N{\mathbb{N}}
\def\Z{\mathbb{Z}}
\def\R{\mathbb{R}}
\def\C{\mathbb{C}}

\def\M{\mathbb{M}}
\def\Om{\Omega}

\def\PH{\Phi}
\def\ol#1{\overline{#1}}
\def\nn{\nonumber}
\def\noop#1{\Vert #1\Vert_{\rm op}}

\def\R{{\mathbb R}}
\def\C{{\mathbb C}}
\def\D{{\mathbb D}}
\def\N{{\mathbb N}}
\def\M{{\mathbb M}}
\def\G{{\mathbb G}}

\def\Z{{\mathbb Z}}

\def\B{{\mathcal B}}

\def\F{{\mathcal F}} 
\def\E{{\mathcal E}}
\def\H{{\mathcal H}}

\def\K{{\mathcal K}}

\def\iy{\infty}

\def\ol#1{\overline{#1}}

\def\no#1{\Vert #1\Vert }

\def\l#1#2{L^{#1}(#2)}

\def\inv{^{-1}}

\def\limk{\lim_{k\to\infty}}

\def\od{\odot}

\def\val#1{\vert #1\vert}

\def\ind#1#2{\rm{ind}_{#1}^{#2}}

\def\L1#1{L^1(#1)}

\def\L#1#2{L^{#1}(#2)}
\def\l#1#2{L^{#1}(#2)}

\def\ts{\times }

\def\lef({\left(}
\def\rig){\right)}

\def\ts{\times }
\def\us#1#2{\underset{#1}{#2}}
\def\wi#1\widehat{#1}

\begin{document}
\title{ The  $C^*$-algebra of the variable Mautner group.}
\author{Regeiba Hedi.}
\address{
Mathematics and Applications Laboratory, LR17ES11, Faculty of Sciences of Gabes University of Gabes,  
Cite Erriadh 6072 Zrig Gabes Tunisia.
} \email{rejaibahedi@gmail.com}

\begin{abstract}
Let $\M=P\times{M}$ be a  variable Mautner group. 
We describe the $C^*$-algebra
$C^*(\M)$ of $\M$ in terms of an algebra of operator fields defined over
$P\times{\C^2} $. 

\end{abstract}
 \maketitle
\section{\bf{Introduction.}}\label{sec:1}
The notion of a variable  group $\G=P\ts{G } $ has been introduced in   
\cite{Le-Lu}. Here  $G $ 
is a locally compact topological Hausdorff space and $P $ is a compact 
Hausdorff space. For every $p\in P $ we have a 
group multiplication  $\underset{p}{\cdot } $ defined on $G $,
such that the mappings 
\begin{eqnarray*}\label{}
 \nn (p;x,y) &\to &
 x\underset{p}{\cdot }y
 \end{eqnarray*}
and
\begin{eqnarray*}\label{}
 \nn x &\to &
 x^{-1_p}
 \end{eqnarray*}
are continuous mappings on $P\times{G\times{G}}\to {G} $ resp. on 
$P\times G\to G $.

We say that the variable group is unimodular if there exists a positive Radon  
measure $dx $ on $G $, which is also a  
left and right Haar measure for the 
groups $G_p:=(G,\underset{p}{\cdot }) $ for every $p\in P $.

A  variable unimodular group defines the family $\left(L^1( 
G_p)\right)_{p\in P}$ 
of  $L^1$-algebras $L^1( G_p)$ with respect to the  Haar measures $d 
x$.
Let $L^1(G)$ be the Banach space $L^1(G,dx)$ of $\C$-valued, 
$dx$-integrable functions on $G$, with norm 
$\displaystyle\no{f}_1=\int_G\val{f(x)}dx.$ Then 
for every $p\in P$ we have  the 
convolution 
\begin{eqnarray*}\label{}
 \nn F\underset{p}{\ast}F'(x)&:= &
 \int_G F(y)F'(y^{-1_p}\underset{p}{\cdot }x)dy,\  F,F'\in\L1 G,\ x\in G.
 \end{eqnarray*}

We denote by $C(P,L^1)$ the involutive Banach algebra  of all continuous 
mappings $F:P\to  \L1 G$ equipped with the product 
\begin{eqnarray*}\label{}
 \nn (F,F') &\to &
 (P\ni p\to F\underset{p}{\ast} F'(p)),\ F,F'\in C(P,L^1)
 \end{eqnarray*}
and the involution
\begin{eqnarray*}\label{}
 \nn F^*(p) &:= &
 F(p)^*,\  p\in P, F,F'\in C(P,L^1).
 \end{eqnarray*}
 The $C^*$-hull the algebra $C(P,L^1) $ of the unimodular  variable group $\G$ will be  
denoted by  $C^*(\G)$. 
As can be seen in \cite{Le-Lu},  the unitary dual or spectrum 
$\widehat{C^*(\G)}$ of $C^*(\G)$ is in bijection 
with the dual space $\widehat \G=\bigcup_{p\in P}\widehat{G_p}$ of $\G$.
Since an element of $\widehat{\G} $ is an equivalence class of irreducible 
unitary representations of $C^*(\G) $, we always choose in such  a 
class $\lbrack 
\pi \rbrack  $  one of its members $\pi $ and we identify $\pi $ with $\lbrack 
\pi \rbrack  $. Define 
then the Fourier 
transform $\F$  on $C^*(\G)$ by
$$\F(c)(\pi)=\pi(c)\in\B(\H_\pi),\ \text{ for all } \pi\in\widehat \G,\ c\in 
C^*(\G).$$
Then $C^*(\G)$
can be identified with the  sub-algebra $\widehat{C^*(\G)} $ of the big 
$C^*$-algebra
$\ell^\iy(\widehat \G)$ of bounded operator fields given by
$$\ell^\iy(\widehat \G)=
\left\{\ph:\widehat \G\longrightarrow\underset{\pi\in\widehat 
\G}{\bigcup}\B(\H_\pi);\ph(\pi)\in \B(\H_\pi), \no 
\ph_\iy:=\underset{\pi\in\widehat \G}{\sup}\noop{\ph(\pi)}<\iy\right\}.$$
Here $\B(\H) $ denotes the space of bounded linear operators on the Hilbert 
space $\H $.

The classical Mautner group $M_\th $ is the semidirect product 
\begin{eqnarray*}\label{}
 \nn M_\th &= &
 \R\ltimes \C^2,
 \end{eqnarray*}
where the reals act on the abelian group $\C^2 $ by
\begin{eqnarray*}\label{}
 \nn t\cdot (a,b) &:= &
 (e^{it}a,e^{i\th t}b), t\in\R, a,b\in \C
 \end{eqnarray*}
 and where $\th $ is a fixed irrational number.
 The group $M_\th $ is a solvable Lie, which is not  type I  and of 
smallest dimension with this property. 

The unitary dual of the group $M_\th $ is therefore  not explicitly known.

 In this paper, we consider $M_\th $ as a limit group of a the family 
$(M_{p,\th}, p\ne 0, )$ of exponential solvable  Lie groups. 
Perhaps this approach can help to understand the structure of $C^*(M_\th) $.

The main result of this paper is the Theorem \ref{slashGis Dstar} where such a 
description of the $C^*$-algebra $C^*(\M)$ as algebra of operator fields 
defined over 
$S:=P\times(\C^2)$. The conditions which determine these  
operator fields are presented in 
Definition \ref{conditions}.
  
 Let me thank  here Professor Jean Ludwig for his suggestions and his 
help during the preparation of this paper.
\section{\bf The variable group $\M$.}\label{variablegroup}
\begin{definition}\label{def1-1}$ $
Let $P,\text{ and } M, $ be  the topological spaces 
\begin{eqnarray*}\label{}
 P&:=&
 [-1,1],\\
 M\nn  & := &
 \R\times\C^2.
 \end{eqnarray*}
Fix  a real number $\th $.
 
 We  define the variable group $\M=\M_\th $,  $\M=(P,M)$ 
by giving for any  $p\in P$ 
 the group multiplication $\cdot_{p}$ on the set $M$ through 
 the formula 
 \begin{eqnarray*}
  (z,w,t)\cdot_{p}(z',w',t')=\left(t+t',e^{(p+i)t}z+z',e^{(-p+i\th)t 
}w+w'\right).
 \end{eqnarray*}
 In particular, if $\th $ is irrational, the groups $M_{\th} $  
are  our Mautner 
groups 
.
Furthermore for any $p\ne0 $, the group $M_p=(M, \cdot_p) $ 
is 
an exponential 
Lie group. The Haar measure of the groups $M_p,\ p\in P, $ is just Lebesgue 
measure $dx $ on 
the space $M=\R\times{\C^2} $.

Its  variable Lie algebra $  P\times\m$ is given by
\begin{eqnarray*}\label{}
 \nn P\times\m &= &
 P\times \R\times \C^2=P\times{ (\R T+\C U+ \C V) }\\
 &&
 T:=
 (1,0,0),\  
U=(0,1,0),\ V=(0,0,1),
 \end{eqnarray*}
and the Lie brackets
\begin{eqnarray*}\label{}  
 \nn \lbrack T,U\rbrack 
 &= &
 (i+p ) U,\\
 \lbrack X, V\rbrack \nn  
&= &
(i  \th -p)V. 
 \end{eqnarray*}
\end{definition}

\subsection{The $C^* $-algebra of $\M $}

\begin{definition}\label{}
\rm 
We let 
  \begin{eqnarray*}\label{}
\nn C_c^\iy(\M)  & = &
\left\{F:P\times{M}\to\C\vert\ 
F\text{ smooth  with compact 
support}\right\},\\
\nn C_c(\M)  & = &
\{F:P\times{M}\to\C\vert\ 
F{\text{  
continuous}}\},\\
 \nn \L1\M &:= &
 C(P,L^1)=\ol{C_c^\iy(\M)}^{\no{.}_1},
 \end{eqnarray*}
where
\begin{eqnarray*}\label{}
 \nn \no{F}_1 &:= &
 \underset{p\in P}{\sup}\no{F(p)}_1, \ F\in C_c^\iy(\M).
 \end{eqnarray*}
Also define the $C^* $-norm $\no{\cdot}_{C^*} $ on $C_c^\iy(\M) $ by 
\begin{eqnarray*}\label{}
 \nn \no F_{C^*} &:= &
 \sup _{p\in P}\no{F(p)}_{C^*},\ F\in C_c^\iy(\M).
 \end{eqnarray*}
This definition gives us  the usual $C^* $-algebra of the  
involutive Banach algebra $\L1\M $, which is given  as the completion
$\ol{\L1\M}^{\no{\cdot}_{C^*}} $ of the algebra $\l1\M $ for the norm 
$\no{\cdot}_{C^*} 
$. 
 \end{definition}
 For $F\in C_c^\iy(\M), \xi\in C_c(M) $ and  $p\in P, $ we then have 
that
 \begin{eqnarray*}\label{}
 \nn F\underset{p}{\ast}\xi(s,c) &= &
 \int_{M}F(p;t,d)\xi((t,d)^{-1_p}\underset{p}{\cdot }(s,c))dtdd\\
 \nn  
&= &
\int_{M}F((p;(s,c)\underset{p}{\cdot }(t,d)^{-1_p})\xi(t,d)dtdd, \ (p;s,c)\in 
\M.
\end{eqnarray*}
This fives us the left regular representation $\LA $ of $\M $ on 
the Hilbert space $\l2M $:
\begin{eqnarray*}\label{}
 \nn \LA^p(F)\xi &:= &
 F\underset{p}{\ast }\xi,  F\in \l1\M,\xi\in\l2M.
 \end{eqnarray*}

\begin{remark}\label{surjofmup}
\rm   
Denote for $p\in P $ the canonical projection $\mu_p:C^*(\M)\to C^*(M_p ) 
$ 
defined by
\begin{eqnarray*}\label{}
 \nn \mu(a)  &:= &
 a(p),\ a\in C^*(\M).
 \end{eqnarray*}
 These mappings $\mu_p  $ are surjective on $\l1\M $. Indeed, if we take any 
$F\in \l1{M_p }$, then we can consider $F $ also as a (constant) element $\ol 
F 
$ of $\l1\M $, where $\ol F(p)=F,\ p\in P, $ and then 
 \begin{eqnarray*}
 \mu_p (\ol F)=F.
 \end{eqnarray*}

The kernel $I_p \cap\l1\M $ in $\l1\M $ of the linear mapping $\mu_p  $ is 
then a 
closed ideal of the 
algebra $\l1\M $ and 
\begin{eqnarray*}\label{}
 \nn \l1\M/I_p  \cap L^1(\M)&\simeq &
 \l1{M_p },
 \end{eqnarray*}
since  $\mu_p  $ is surjective. 
Hence we also have that the mapping 
\begin{eqnarray}\label{muis surj}
  \mu_p :C^*(\M) &\to &
 C^*(M_p ),\ p\in P, 
\end{eqnarray}
  is surjective.
 \end{remark}

 \subsection{The dual space of $\M $}\label{dual bbM}
 
$ $

It is well known (see \cite{Le-Lu}) that the spectrum 
$\widehat{\M} $ of the variable group 
$\M=P\times M $ is given by
\begin{eqnarray*}\label{}
 \nn \widehat{\M} &= &
 \bigcup_{p\in P}\widehat{M_p }.
 \end{eqnarray*}

 \begin{definition}\label{pipldef}
\rm  Let for $(p, \ell)\in P\times{\C^2},  $
\begin{eqnarray*}\label{}
 \nn \pi^p _\ell &:= &
 \ind{C}{M_p }\ch_\ell,
 \end{eqnarray*}
where as before $C=\{ 0\}\times{\C^2}\subset M $ and where
$\ch_\ell(c)=e^{i \langle c,\ell\rangle },\ c\in C $.
 \end{definition}
The Hilbert space $\H_\ell^p $ of the representation $\pi_\ell^p $ 
is the space of functions 
\begin{eqnarray*}\label{}
 \nn \H_\ell^p &:= &
 \{ \xi: M\to\C;\ \xi \text{ measurable},\\
 \nn  
& &
\xi(m\underset{p}{\cdot }c)=\ch_\ell(-c)\xi(m); m\in M, c\in C,\\
\nn  
& &
\no\xi_2^2=\int_{\R}  \vert{\xi(t,0,0) }\vert ^2dt<\iy\}\\
\nn  
&\simeq &
\l2\R.
 \end{eqnarray*}
The group $M_p $ acts then by left translation on this space.
\begin{definition}\label{cstarfield}
\rm   For $a\in C^*(\M) $
define the operator field 
$\widehat a $ on $P\times\C^2 $ by
\begin{eqnarray*}\label{}
 \nn \widehat a(p, \ell) &:= &
 \pi^p_\ell(a(p)),\ \ell\in\C^2,p\in P.
 \end{eqnarray*}

 \end{definition}
Let us compute for $ p\in P$, $ \ell\in (\C^2)^*$ and $ F\in\l1{M_p }$ the 
operator $ \pi^p _\ell(F)\in \B(\l2\R)$. 
For $\ell=(f,g)\in(\C^2)^*\simeq\C^2$, $m=(t, c)\in M $ and 
$\xi\in\l2\R, s\in\R $ 
we have 
that 
\begin{eqnarray}\label{reppipll}
  \pi^p _\ell(t,c)\xi(s) &= &
 \xi((t,c)^{-1_p}\underset{p}{\cdot } (s,(0,0))\\
 \nn  
&= &
\xi((s-t, -(t-s)\underset{p}{\cdot }c)(s,0,0))\\
\nn  
&= &
e^{i\langle (t-s)\underset{p}{\cdot } c,\ell\rangle}\xi(s-t).
 \end{eqnarray}
Hence
\begin{eqnarray}\label{ind ell}
  \pi^p _\ell(F)\xi(s) &= &
 \int_M F(g)\pi^p _\ell(g)\xi(s)\\
 \nn  
&= &
\int_{\R\times{\C^2}} F((t,c))e^{i\langle (t-s)\underset{p}{\cdot } 
c,\ell\rangle}\xi(s-t)dcdt\\
\nn  
&= &
\int_{\R\times{\C^2}} F((s-t,c))e^{i\langle (-t)\underset{p}{\cdot } 
c,\ell\rangle}\xi(t)dcdt\\
\nn  
&= &
\int_{\R\times{\C^2}} F((s-t,c))e^{i\langle  
c,(-t)\underset{p}{\cdot }\ell\rangle}\xi(t)dcdt\\
\nn  
&= &
\int_\R \widehat F^2(s-t,t\underset{p }{\cdot}\ell)\xi(t)dt,
 \end{eqnarray}
where 
\begin{eqnarray*}\label{}
 \nn t\underset{p}{\cdot }(f,g) &:= &
 (e^{(-i+p)t)}f,e^{(-i\th-p)t)}g),\\
 \nn  \widehat F^2(s,\ell)
&:= &
\int_{{\C^2}} F(s,c)e^{i\langle  
c,\ell\rangle}dc, \ \ell=(f,g)\in\C^2,\ p\in 
P,\ t\in\R.
 \end{eqnarray*}
 
\begin{remark}\label{someirre}
\rm   We see 
\begin{itemize}\label{}
\item that for  $p\ne 0,\ \ell=(\ell_1,\ell_2) $ with $\ell_1\cdot \ell_2\ne 
0 $, the operators 
$\hat 
a(p, \ell) $ are compact and that for $\ell\ne(0,0) $,  the representations
\begin{eqnarray*}\label{}
 \nn C^*(\M)\to \B(\l2\R):\  
 a\to \hat a(p, \ell), 
 \end{eqnarray*}
are irreducible, since $M_p  $ is then an exponential Lie group, 
\item that for $p=0$ with $\ell=(f,g),\ f\cdot g\ne0 $  the 
representations $\pi_{\ell}^p  $ are irreducible too,
 since then the stabilizer of $\ell $ is the subgroup $C $. 
 \item that for $p=0 $ with $\ell=(f,0),\ f\ne0$ or $\ell=(0,g), g\ne 0, $ 
the 
representations $\pi_{\ell}^p  $ are not irreducible,
 since then the stabilizer of $\ell \in\C^2$ is the subgroup 
 \begin{eqnarray*}\label{}
 \nn S_{\ell}^{0} &= &
 \{ (m,c)\vert\  m\in 2\pi \Z,\ c\in C\}.
 \end{eqnarray*}
 resp.
 \begin{eqnarray*}\label{}
 \nn S_{\ell}^{0} &= &
 \left\{ (m,c)\vert\ m\in \frac{2\pi}{\th}\Z,\ c\in C\right\}.
 \end{eqnarray*}
Hence the representations
\begin{eqnarray*}\label{}
 \nn \pi^{0}_{\omega,\ell} & :=&
 \ind{S_\ell^0}{M_\th}\ch_{\omega,\ell}
 \end{eqnarray*}
where $\om\in \R/ \Z $ (resp. in $\R/{2\pi\th \Z}) $ and 
\begin{eqnarray*}\label{}
 \nn \ch_{\omega,\ell}(m,c) &:= &
 e^{ i (m\om+ \langle c,\ell\rangle )},\ m\in2\pi\Z,\ c\in C,\\
& &
\text{resp. } m\in \frac{2\pi}{\th}\Z, c\in C,
 \end{eqnarray*}
  are irreducible.

 \item  If $\ell=(0,0) $ and $p\in P $, then the representation 
 $\pi^p _{(0,0)}$ is the left regular representation of $M_p  $ on 
$\l2{M_p /C }$ and therefore equivalent to the direct integral over the space 
of 
unitary characters $Ch^{p}:=\{ \ch_{r},\ r\in\R\} $ of $M_p $, where
\begin{eqnarray*}\label{}
 \nn \ch_{r}(t,c) &:= &
 e^{i tr },\ (t,c)\in M.
 \end{eqnarray*} 

\end{itemize}
 \end{remark}
 
\subsection{Desintegration of the induced representation.}$ $

We define the Fourier Transform $ \F:\l2M \to \l2{\R\times \C^2}$ by
\begin{eqnarray*}\label{}
\nn \F(\xi)(\ell)(s)&:=&
 \int_{\C^2}\xi(s,c)e^{i\langle c,\ell\rangle }dc.
 \end{eqnarray*}
 We can disintegrate the left regular representation $ (\LA^p,\l2{M_p })$  of $ 
M_p $  
into a direct integral of irreducibles   using 
the Fourier transform $ \F$ .
Indeed, we have for $ \xi\in C_c(M_p )$ and $ \ell\in\C^2$ that the function
\begin{eqnarray*}\label{}
 \nn s\mapsto \F(\xi)(p, s,\ell),\ (s\in\R)
 \end{eqnarray*}
is contained in the Hilbert space of the representation $ 
\pi^p _{\ell}$.

\begin{definition}\label{lonec def}
  Define the subspace $\l1{M}^c$ of $\l1{M} $ by
\begin{eqnarray*}\label{}
 \nn \l1{M}^c &:= &
 \left\{ F\in\l1{M}  \text{ such that } \widehat F^2 \in C_c(\R\times\C^2)\right\}.
 \end{eqnarray*}

 \end{definition}

We take now $ p\in P$ and some $ \xi,\et\in \l2M$. Then for 
$ m=(t,c)\in M,$ we have that (forgetting about constants like the number $\pi 
$):
\begin{eqnarray*}\label{}
 \nn \langle \LA^p (m)\xi,\et\rangle  &= &
 \int_M \xi(m\inv\underset{p }{\cdot}  (u,b)\ol{\et(u,b)}dudb\\
 \nn  
&= &
\int_\R\int_{\C^2} \xi(u-t, -((t-u)\cdot c)+b )\ol{\et(u,b)}dudb\\
&\overset{\text{Plancherel}}{=}&
\int_\R\int_{\C^2} e^{- i \langle (u-t)\cdot c,\ell\rangle }\hat \xi(u-t, 
\ell)\ol{\hat\et(u,\ell)}dud\ell\\
\nn  
&=&
\int_{\C^2}\langle \pi^p _\ell(g)(\hat \xi(\ell)),\hat\et(\ell)\rangle_2 
d\ell.
 \end{eqnarray*}

 This shows that 

 \begin{eqnarray}\label{lag dec}
\nn\LA^p &= &
\oint_{\C^2}\pi^p _\ell d\ell,
 \end{eqnarray}
 
\begin{definition}\label{starreg}
\rm   Let $A $ be an involutive, semi-simple Banach algebra and let $C^*(A) $ 
be 
its $C^* $-algebra. We say that $A $ is \textit{star regular}, if the canonical 
mapping from
$Prim(C^*(A))\to Prim^*(A) $ defined by
\begin{eqnarray*}
 I\to I\cap A
 \end{eqnarray*}
is a homeomorhism for the Jacobson topologies.
 \end{definition}

It had been shown in \cite{L-S-V} that every connected locally compact group $G 
$ of polynomial group (in particular the group $M_\th $) is $* $-regular. This 
means that the group algebra $\l1G $ is star regular.
In particular this implies that 
the closure for the Jacobson topology of $I_1=ker_{\l1{M_\th}} (\pi)\in 
Prim^*(M_\th) $ 
is the primitive ideal $ker_{C^*({M_\th })}(\pi) $. 

\subsection{The topology of the orbit space.}$ $

  We define  the Kirillov-Pukanszky mapping $\K: \E(\M)\to Prim(\M)  $  by
  \begin{eqnarray*}\label{}
 \nn \K(\E_\ell^p) &:= &
 (I^p _\ell),
 \end{eqnarray*}
where for $p\in P$ and $\ell\in \C^2 $ or $\ell=s\in\R $ the symbol 
$I_\ell^p =\text{ker}(\pi^p _\ell) $ is the 
primitive ideal given by the  quasi-orbit of $\ell $.

\begin{theorem}\label{star regular variable}
The variable group $\M $ is $* $-regular.
 \end{theorem}
\begin{proof} It  suffices to apply the proof of Theorem 4
in \cite{BLSV}, since the mapping $\E^p _{\ell}\longrightarrow I^p _{\ell}$ 
from 
$\E(\M)$ into 
$\text{Prim}(M_p ), p\in P,$ is bijective.

\end{proof}

\begin{theorem}\label{kirillov variable}
The Kirillov-Pukanszky mapping $\K $ is a homeomorphism of the quasi-orbit 
space $\E(\M)$ onto $\text{Prim}(\M) $.
 \end{theorem}
\begin{proof} Let $(p_k,(u_k,f_k,g_k))_k $ be  a converging sequence in 
$(P\times{\m^*}) $ with limit point $(p, (u,f,g) )$. Take $F\in\l1\M $. 
Then $F(p_k)\in\l1{\R\times{\C^2}} $ converges in $L^1 $-norm to $F(p) $. 
 
But then for almost every $u\in\R $ we have that the $\l1{\C^2} $-limit of the 
sequence $(F(p_k, u)) $ is the $L^1 $-function $F(p, u) $.
Hence the $L^\iy(\C^2) $ sequence $(\widehat F^2(p_k,u)) $ converges 
uniformly to $ 
(\widehat F^2(0,u)). $ This shows that for any infinite subset $S\subset \N 
$ and every $F\in \underset{s\in S}{\bigcap}I^{p_s}_{ (u,f_s,g_s)} $  we have that 

\begin{eqnarray*}\label{}
 \nn \widehat F^2(p,  u, (M_p )\us p\cdot (f,g))=\{ 0\},
 \end{eqnarray*}
which means that $F\in I_{ (f,g)^p }
$. We have shown that   
$\underset{s\in S}{\bigcap}I^{p_s}_{ (f_s,g_s)}\subset I^{p}_{ (f,g)} $.

If $(f,g)=0 $, then we use the sequence of numbers 
\begin{eqnarray*}\label{}
 \nn \left(\int_\R \widehat F^2(p_k, (v, 
f_k,g_k))e^{-2\pi i uv}dv\right)_k,
 \end{eqnarray*}
 which converges to $\widehat 
F^{1,2}(p, u,0,0) $.

This 
again shows that $\underset{s\in S}{\bigcap}I_{(u_s,f_s,g_s)}^{p_s}\subset I^{p}_{ 
(u,0,0)} $.

It suffices now to use the fact that the variable group $\M $ is $*$-regular.

For the other direction we apply \cite{Lud1} Theorem 4.1. resp. the proof of 
Theorem \ref{kirillov variable}.

\end{proof}

\section{\bf Limit conditions.  }\label{limitcondition}
 
\begin{definition}\label{rpk def} 
   \end{definition}
 For  $C>0 $ let 
\begin{eqnarray*}\label{}
 \nn M_C & =&
M_{ \vert{\cdot  }\vert \leq C}
 \end{eqnarray*}
be the multiplication operator on $\l2\R $ with the characteristic function of 
the interval $[-C,C] $.

Choose for any $\de\in [0,1] $ an $\ve_\de>0 $ such that the function $\de\to  
\frac{\de}{\ve_\de} $ is decreasing with limit 0 and such that 
$\lim_{\de\to 0}\frac{\de}{\ve_\de}=0 $. 
 
Let 
\begin{eqnarray}\label{rpk def}
 \nn R^\de_j &:= &
 \frac{j}{\ve_\de}, j\in\Z,\ \de\in[-1,1]\setminus \{ 0\}. 
\end{eqnarray}

 For any $\de\in [0,1]\setminus \{ 0\} $ and $j\in \Z ,$ let 
 $M^\de_j $ be the multiplication operator of $\l2\R $ with the characteristic 
function of the interval 
\begin{eqnarray*}\label{}
 \nn 
I_\de(j):=\left[R^\de_{j},R^\de_{j+1}\right]=\left[\frac{j}{\ve_\de},\frac
{j+1}{\ve_\de}
\right]=I_\de(0)+R^\de_j
.
 \end{eqnarray*}
 Let now $\de\to r(\de)$ be a positive function defined on $[0,1]\setminus \{ 0
\}$, 
such that 
$\lim_{\de\to 0 }r(\de)=+\infty, \lim_{\de\to 0} \frac{r(\de)}{R^\de_1}=0$.  
Then for any $\de\in [0,1]\setminus \{ 0\} $ and for all $j\in \Z $, let 
 $N^\de_j $ be the multiplication operator of $\l2\R $ with the characteristic 
function of the interval  
\begin{eqnarray*}\label{}
 \nn J_\de(j):=\left[R^\de_{j}-r(\de),R^\de_{j+1}+r(\de)\right], j\in\Z.
 \end{eqnarray*}
 
Let for $z=t+iu\in\C,p\in P, $ and $\ell=(\ell_1,\ell_2)\in\C^2 $
\begin{eqnarray*}\label{}
 \nn z\underset p \od \ell &:= &
 (e^{p t+iu}\ell_1,e^{-p t+i\th u}\ell_2).
 \end{eqnarray*}

Let for $a\in C^*(\M)  $, $p\in P,\de\in [0,1]\setminus \{ 0\}$ and  $\ell\in 
\C^{2} 
$:
\begin{eqnarray}\label{sipl def}
  \si^{p} _{\ell}(a) &:= &
 \sum_{j\in\Z}M^p_j\circ  
\pi^{0}_{R^p\underset p \od \ell}(a)\circ  
N^\de_j.
 \end{eqnarray}
 We observe that for any $a\in C^*(\M),p\in P\setminus \{ 0\},  $ and 
$\xi\in\l2\R $, we have 
that
 \begin{eqnarray*}\label{}
 \nn  \Vert{ \si^{p} _{\ell}(a)\xi}\Vert_{ 2} ^2 &= &
 \sum_{j\in\Z} \Vert{M^p_j }\circ  
\pi^{0}_{R^p_j\underset p \od \ell}(a)\circ  
N^p_j(\xi)\Vert _2 ^2\\
\nn  
&\leq &
  \no a_{C^*}^2\left(\sum_{j\in\Z}\Vert  
N^p_j(\xi)\Vert _2 ^2\right)\\
\nn  
&\leq &
\no a_{C^*}^2\left(\sum_{j\in\Z}  
\Vert (M^p_{j-1}+M^p_j+M^p_{j+1})(\xi)\Vert  _2 ^2\right)\\
\nn  
&\leq &
9\no a_{C^*}^2\Vert \xi\Vert _2 ^2.
 \end{eqnarray*}

\begin{proposition}\label{sikminuspiO}
For any $a\in C^*(\M)  $, we have that
\begin{eqnarray*}\label{}
 \nn \lim_{\de\to 0} \sup_{\underset{ \vert{p }\vert \leq \de}{\ell\in\C^2}}
 \noop{\si^{p} _{\ell}  (a)-\pi^{p} _{\ell}(a)} &= &
 0.
 \end{eqnarray*}

 \end{proposition}
\begin{proof} 
There exists for any $F\in\l1\M $, for which $\widehat F^2\in C_c^\iy(\R\times 
\C^2),
$ some $\va\in C_c(\R,\R_{\geq 0}) $ and $\ps\in C_c(\C^2\times{\C^2},\R_{\geq 
0}) $, such 
that
\begin{eqnarray*}\label{}
 \nn \vert \widehat F^2(u,\ell)-\widehat F^2(u,\ell')\vert   &\leq  
&
 \va(u)\vert{\ell-\ell' }\vert \ps(\ell,\ell'),\ u\in\R,\ \ell,\ell'\in\C^2.
 \end{eqnarray*}
  Now, since $\va$  and $\ps $ are compactly supported, there exists $C>0 $ 
such that 
 $\va(u)=0 $ whenever $ \vert{u }\vert >C $ (resp. $\ps(\ell,\ell')=0 $ if $ 
\vert{\ell }\vert >C$ and $ 
\vert{\ell' }\vert >C$). Hence for any $j\in \Z $, 
 if $u\in I_\de(j) $ then $\va(u-t) \ne 0$ implies that for  $\de=p-p_0$ 
small 
enough we have that $t\in I_\de(j-1)\cup I_\de(j)\cup I_\de(j+1) $.

Therefore
\begin{eqnarray*}\label{}
 \nn M_j^p\circ  \pi^p _{\ell}(F) &= &
 M_j^p\circ  \pi^p _{\ell}(F)\circ N^p_j,\ \ell\in \C^2,\ p\in P.
 \end{eqnarray*}
 
We then have  for $\ell=(\ell_1,\ell_2)\in 
\C^2 $, 
that
\begin{eqnarray}\label{sip-pi0}
& &\\
 & &
 \nn(\si^{p} _{\ell}(F)-\pi^{p} _{\ell}(F))(\xi)(u)\\
\nn&=&
\sum_{j\in\Z}M^p_j\circ  \pi^0_{R^p_j\underset p \od \ell}(F)\circ  
 N^p_j-(\sum_{j\in\Z}M^p_j)\circ  \pi^{p} _{\ell}(F)(\xi)(u)\\
\nn&=&
\sum_{j\in\Z}(M^p_j\circ  \pi^0_{R^p_j\underset p \od \ell}(F)-\pi^{p} 
_{\ell}(F))(\xi)(u)\\
\nn &=&
\sum_{j\in\Z}M^p_j\circ(\pi^0_{ (e^{p R^p_j}\ell_1, e^{-p 
R^p_j}\ell_2)}(F)-\pi^{p} _{(\ell_1,\ell_2)}(F))(\xi)(u)\\
\nn &= &
\sum_{j\in\Z} 1_{I_\de(j)}(u)\Big(\int_{ \R} (\widehat 
F^2(u-t,(R^p_j+it)\underset p{\od}\ell)-\\
\nn  
&- &
\widehat F^2(u-t,(t+it)\underset p{\od}\ell)\xi(t)dt\Big).
 \end{eqnarray}

Furthermore for any $\ell=(\ell_1,\ell_2)\in\C^2$, $p\in P\setminus \{ 0\}$, $ 
t,u\in I_\de(j)$ and $j\in\Z $, we have  that
\begin{eqnarray}\label{et-eR}
 & &\sup_{t\in I_\de(j)}\vert(  e^{(p R^p_j+it} \ell_1,e^{(-p R^p_j+i\th t} 
\ell_2) -(e^{(p+i)t}\ell_1,e^{(-p+i\th) t} \ell_2)\vert  \\
\nn  
&\leq &
\underset{t\in I_\de(j)}{\sup}
(\vert e^{pR^p_j}-e^{pt}\vert  \vert{\ell_1 }\vert +\vert 
e^{-pR^p_j}-e^{-pt}\vert 
 \vert{\ell_2 }\vert )\\
\nn  
& \leq&
\underset{t\in I_\de(0)}{\sup}
(\vert e^{p(t-R^p_1)}-1\vert e^{pR^p_j} \vert{\ell_1 }\vert +\vert 
e^{p(R^p_1-t)}-1\vert 
e^{-pR^p_j} \vert{\ell_2 }\vert ).
 \end{eqnarray}
Therefore
 \begin{eqnarray*}\label{}
 \nn \nn  
& & 
\vert \widehat 
F^2(u-t,(R^p_j+it)\underset{p}{\od} \ell)-
\widehat F^2(u-t,
(t+it)\underset{p}{\od}\ell)\vert \\
\nn  
&\leq  &\underset{t\in I_\de(j)}{\sup}
(\vert (e^{pR^p_j}-e^{pt})\ell_1\vert +\vert (e^{-pR^p_j}-e^{-pt})\ell_2\vert 
)\va(u-t)\ps(e^{pR^p_j+it}\ell_1, e^{-pR^p_j+i\th t}\ell_2)\\
\nn  
&\overset{(\ref{et-eR})}{\leq}  &
(\underset{t\in I_\de(0)}{\sup}
(\vert e^{p(t-R^p_1)}-1\vert e^{pR^p_j}\vert \ell_1\vert +\vert 
e^{p(R^p_1-t)}-1\vert 
e^{-pR^p_j}\vert \ell_2\vert ))\va(u-t) 
\ps(({pR^p_j+it})\underset{p}{\od }\ell)\\
\nn  
&\leq  &K  pR^p_1\va(u-t)( 
e^{pR^p_j}\vert \ell_1\vert+e^{-pR^p_j} \vert{\ell_2 }\vert )\ps(e^{it}e^{p 
R^p_j} \ell_1,e^{i\th t}e^{-p R^p_j} \ell_2)) \\
\nn  
&\leq &
\frac{KC_F\de}{\ve_\de}\va(u-t) , 
t,u\in\R,
 \end{eqnarray*}
for some $K>0, C_F>0.$

Hence we have that 
 \begin{eqnarray}\label{est sip-pi0}
  \nn& &
 \no{(\si^{p} _{ (\ell_1, \ell_2)}(F)-\pi^{p} _{(\ell_1, \ell_2)}(F))(\xi)
}_2^2\\
&= &
\nn\sum_{j\in\Z}\no{M^p_j(\si^{p} _{ (\ell_1, \ell_2)}(F)-\pi^{p} _{(\ell_1, 
\ell_2)}(F))(\xi)}^2\\
\nn  
\nn&= &
\sum_{j\in\Z}\no{M^p_j(\si^{p,\de} _{ (\ell_1, \ell_2)}(F)-\pi^{p} _{(\ell_1, 
\ell_2)}(F))(M^p_{j-1}+M^p_{j}+M^p_{j+1})(\xi)}^2\\
\nn&\leq&
(KC_F\frac{\de}{\ve_\de}\no\va_1)^2 
(\sum_{j\in\N}\no{M^p_j(\xi)}_2^2)\\
\nn&=&
9(KC_F\frac{\de}{\ve_\de}\no\va_1)^2 \no{\xi}_2^2,\ \xi\in\l2\R.
 \end{eqnarray}

This shows that
\begin{eqnarray*}\label{}
 \nn \lim_{\de\to 0} \sup_{\underset{ \vert{p }\vert \leq \de}{\ell\in\C^2}}
 \noop{\si^{p} _{\ell}  (a)-\pi^{p} _{\ell}(a)} &= &
 0.
 \end{eqnarray*}

The statements  then  is  a consequence of the density in $\l1M $ of the 
subspace generated by the 
set $\left\{ F\in\l1M\vert\  \widehat F^2\in C_c^\iy(\R\times{\C^2})\right\} $.
\end{proof}

\begin{corollary}\label{limpikis0}
 
For any $a\in C^*(\M) $, such that $a(0)=0 $, 
we have that
\begin{eqnarray*}\label{}
 \nn 
\lim_{p\to 0}\noop{\LA^p (a)} &= &
 0
 \end{eqnarray*}
 \end{corollary}
    \begin{proof} This follows from the fact that
    \begin{eqnarray*}\label{}
 \nn \no a &= &
 sup_{\ell\in\C^2}\noop{\pi^p_\ell(a)}, a\in C^*(M_p),p\in P.
 \end{eqnarray*}

\end{proof}
                                                             
 \subsection{$p_0 $ different from $ 0$}$ $

We observe that for $p,p_0 $ different from $ 0$, the groups 
$M_p $ and $M_{p_0} $ are isomorphic. Indeed the mappings $h_{p_0,p} $ 
defined by 
\begin{eqnarray*}\label{}
 \nn h_{p_0,p}(s,c) &:= &
 (\frac{p_0}{p}s,c), (s,c)\in M,
 \end{eqnarray*}
is an isomorphism from $M_{p_0} $ onto $M_p $ since for 
$(s,c), (s',c')\in M $ we have that
\begin{eqnarray*}\label{}
 \nn h_{p_0,p}((s,c)\underset{p_0}{\cdot }(s',c')) &= &
 h_{p_0,p}(s+s',e^{-p_0s' }\cdot c+c')\\
 \nn  
&= &
(\frac{p_0}{p}(s+s'),e^{-p_0s' }\cdot c+c'),
 \end{eqnarray*}
and
\begin{eqnarray*}\label{}
 \nn h_{p_0,p}(s,c)\underset{p}{\cdot }h_{p_0,p}(s',c') &= &
 (\frac{p_0}{p}s,c)\underset{p}{\cdot }(\frac{p_0}{p}s',c')\\
 \nn  
&= &
(\frac{p_0}{p}(s+s'),e^{-(p\frac{p_0}{p}s')}c+c'') \\
\nn  
& =&
(\frac{p_0}{p}(s+s'),e^{-p_0s'}c+c').
 \end{eqnarray*}

Take $F\in \l1{\M}, p\in P\setminus\{ 0\} $, such that the 
function 
\begin{eqnarray*}\label{}
 \nn (p,s,c) &\to &
 \widehat F^2(p,s,\ell)
 \end{eqnarray*}
is contained in $C_c^\iy(P\times\R\times{\C^2}) $.
Then
\begin{eqnarray*}\label{}
 \nn \left(\pi_\ell^p(F)-\pi^{p_0}_\ell(F)\right)(\xi)(s) &= &
 \int_\R \left(\widehat F^2(p,s-t,t\underset{p}{\cdot }\ell)-\widehat F^2(p_0, 
s-t,t\underset{p_0}{\cdot }\ell)\right)\xi(t)dt
 \end{eqnarray*}
Therefore 
\begin{eqnarray*}\label{}
 \nn \lim_{\de\to 0} \sup_{\underset{ \vert{p }\vert \leq \de}{\ell\in\C^2}}
 \noop{\pi^{p_0} _{\ell}  (a)-\pi^{p} _{\ell}(a)} &= &
 0.
 \end{eqnarray*}
We see in this way that if for some $a\in C^*(\M) $ and $p_0\in P\setminus \{ 
0\} $ we have that 
\begin{eqnarray*}\label{}
 \nn a(p_0) &= &
 0
 \end{eqnarray*}
then
automatically
\begin{eqnarray}\label{ptopo}
  \lim_{p\to p_0}\no{a(p)-a(p_0)} &= &
 0.
 \end{eqnarray}

\subsection{A $C^*$-condition.}

\begin{definition}\label{sikdef}
\rm   Let $CB(P) $ be the algebra of operator fields 
$(\Ph(p, \ell)\in\B(L^2(\R)))_{p\in P, \ell\in\C^2} $, for which the mapping
\begin{eqnarray*}\label{}
 \nn  (p,\ell) &\to&\Ph(p, \ell) \\
& &
\text{is strongly continuous on }P.
 \end{eqnarray*}
For 
an operator field $\Ph\in CB(P) $ let
\begin{eqnarray}\label{sipl def}
 \nn \si^{p} _{\ell}(\Ph) &:= &
 \sum_{j\in\Z}M^p_j\circ  
\Ph(p,R^p_j\underset p \od \ell)\circ  
N^p_j,P\ni p\ne 0,\\
\nn  \si^p(\Ph)
&:= &
\oint_{\C^2}\si^p_\ell(\Ph(p,\ell))d\ell, p\in P,\\
\nn  \Ph(p)
&:= &
\oint_{\C^2}\Ph(p,\ell)d\ell, p\in\ell.
 \end{eqnarray}
 \end{definition}

\begin{remark}\label{siisthatsi}
\rm   Let $a\in C^*(\M) $. Then, as we have seen above,
the operator field $(\widehat a(p,\ell) )_{p\in P, \ell\in \C^2}$ is contained in 
$CB(P) $.

 \end{remark}
\begin{definition}\label{conditions}
 Let $\D^*(\M)=\D^*$ be the space consisting of all 
operator  fields $\Ph $ defined over $S:=P\times\C^2 
$  and contained in $CB(P) $ such that:
 \begin{enumerate}
 \item  $\Ph(p)\in \widehat{C^*(M_p )},\ p\in P $.
\item The field $\Ph $ satisfies the condition: for some $p_0\in P$ we have 
that 
\begin{eqnarray*}
 \lim_{p\to 0} \noop{\si^{p}_\ell(\Ph)-\Ph(p,\ell)}=0.
\end{eqnarray*}
\item  For $p_0\ne 0 $ in $P $, we have that
\begin{eqnarray*}\label{}
 \nn \lim_{p\to p_0}\no{\Ph(p)-\Ph(p_0)}&= &
 0.
 \end{eqnarray*}

 \item  the same conditions are satisfied by the field $\Ph^* $.
 \end{enumerate}
 The $C^* $-algebras $C^*(M_p), p\ne 0, $ have been characterized 
as algebras of operator fields in the paper \cite{Lin-Lud}.
 \begin{remark}\label{ph0}
{\rm It follows from Proposition \ref{sikminuspiO} and from (\ref{ptopo}) that 
for every 
$\Ph\in\widehat{C^*(\M)} 
$, the operator field $\Ph $ is contained in $\D^*(\M) $}.
 \end{remark}

\end{definition}

\begin{definition}\label{cPactiononCstar}
\rm   Let $a\in C^*(\M)  $ and $\va\in C(P) $. Define $\va\cdot a $ by
\begin{eqnarray*}\label{}
 \nn \va\cdot a(p) &:= &
 \va(p)a(p), p\in P.
 \end{eqnarray*}
 For an operator field $(\Ph(p,  \ell)\in \B(\l2\R))_{\ell\in\C^2,p\in P} 
,$ 
and $\va\in C(P) $, let 
\begin{eqnarray*}\label{}
 \nn (\va\cdot \Ph)(p, \ell) &:= &
 \va(p)\Ph(p, \ell),\ \ell\in\C^2,\ p\in P.
 \end{eqnarray*}

 \end{definition}
\begin{proposition}\label{vaaincstar}
Any $\va\in C(P) $ defines  a central multiplier  of $\D^* $.
 \end{proposition}
\begin{proof} Clearly, for every $\Ph\in 
D^*(\M) $ and any $\va\in C(P) $, the new operator field also satisfies  the 
conditions (1) to (4), since  $\si^{p} (\va\cdot
\Ph)=\va(p)\cdot 
\si^{p} (\Ph), p\in P,\ell\in \C^2 $.

\end{proof}

\begin{theorem}\label{slashGis Dstar}
The space $\D^*(\M)$ is  a $C^*$-algebra, which is isomorphic to $C^*(\M) $.
 \end{theorem}
\begin{proof} 
It is clear that $\D^* $ is an involutive Banach space. 
Let us show that condition $(2) $ is stable under the composition of operator 
fields. Let $\Ph,\Ph'\in \D^* $. 

Then $\Ph(0) $ and $\Ph'(0)$ are contained in $C^*(M_{0}) $.  Since the mapping 
$C^*(\M)\to \C^*(M_p ) $ is 
surjective for any $p\in P $, we  find 
$b,b'\in C^*(\M)  $, 
such that $\Ph(p)=\widehat b(p), \Ph'(p)=\widehat{b'}(p) $. 
Condition $(2) $ then says that 
\begin{eqnarray*}\label{}
 \nn & &\lim_{p\to 0}\noop{ \Ph\circ 
\Ph'(p)-\si^{p} (\Ph \circ  \Ph')} \\
&=&
\lim_{p\to p_0}\noop{\Ph\circ 
\Ph'(p)-\si^{p} (\widehat b\circ  \widehat{b'})}\\
&=&
\lim_{p\to p_0}\noop{\Ph\circ 
\Ph'(p)-\si^{p} (\widehat {b\cdot b'})}\\
&=&
\lim_{p\to p_0}\noop{\Ph\circ 
\Ph'(p)- \widehat {b\cdot b'}(p)}\\
&=&
\lim_{p\to p_0}\noop{\Ph\circ 
\Ph'(p)-\widehat {b}(p)\circ  
\widehat{b'}(p)}\\
&=&
 \lim_{p\to p_0}\noop{\Ph\circ 
\Ph'(p)-\si^{p}  (\widehat b)\circ 
\si^{p}  (\widehat{b'})}
\nn \\
&=&
 \lim_{p\to p_0}\noop{\Ph(p)\circ 
\Ph'(p)-\si^{p}  (\Ph)\circ 
\si^{p}  (\Ph')}\\
&= &
0.
 \end{eqnarray*}
Thus $\D^* $ is a $C^* $-algebra.

Let us show that $\D^*=C^*(\M)  $.
Let $\pi\in \widehat{\D^*}$. Then there exists a  character $\ps $ of the 
algebra $C(P) $, such that 
\begin{eqnarray*}\label{}
 \nn \pi(\va \cdot \Ph) &= &
 \ps(\va)\pi(\Ph),\ \va\in C(P),\ \Ph\in \D^*.
 \end{eqnarray*}
Now every character of the algebra $C(P) $ is given by some point 
evaluation. Hence there exists $p_\pi\in P $ such that
\begin{eqnarray*}\label{}
 \nn\nn \pi(\va \cdot \Ph) &= &
 \va(p_\pi)\pi(\Ph),\ \va\in C(P),\ \Ph\in \D^*.
 \end{eqnarray*}
Let \begin{eqnarray*}\label{}
 \nn J_\pi &:= & J_{p_\pi}=
 \{ \Ph\in \D^*\vert\ \Ph(p_\pi)=0 \}.
 \end{eqnarray*}
Then $J_\pi $ is a closed ideal of $\D^* $. 

Choose  a sequence bounded by 1 of 
functions $(\ps_k)_k\subset C([0,1]), $ such that $\ps=0 $ on  $ 
U_k:=\{ p'; \vert p'-p_\pi\vert  \leq \frac{1}{k}, \}$ 
and $\ps_k(p)= 1,p\not\in U_{2k},\  
k\in\N$. 
Then
for any $\Ph\in J_{p_\pi} $ we have that
\begin{eqnarray*}\label{}
 \nn \limk \ps_k \cdot \PH&= &
 \Ph,
 \end{eqnarray*}
 since 
 \begin{eqnarray*}\label{}
 \nn \no{\ps_k\cdot \Ph-\Ph} &= &
 \sup_{p\in P}\no{\ps_k(p)\Ph(p)-\Ph(p)}_{op}\\
 &= &
 \sup_{\vert p\vert \leq 
\frac{2}{k}}\no{\ps_k(p)\Ph(p)-\Ph(p)}_{op}\\
 \nn  
&\leq &
2\sup_{\vert p-p_\pi\vert \leq \frac{2}{k}}\no{\Ph(p)}_{op}\\
\nn  
&\to &
0
\end{eqnarray*}
since for this $\Ph $ we have that $\lim_{p\to 0}\Ph(p)=0 $ by condition (2) 
if $p_\pi=0 $ or by condition (3) if $p_\pi\ne 0 $.
This shows that 
\begin{eqnarray*}\label{}
 \nn \pi(\Ph) &= &  
 \lim_{k\to\iy} \ps_k(p_\pi)\pi(\Ph)\\
 &=&
 0.
 \end{eqnarray*} 
Hence
\begin{eqnarray*}\label{}
 \nn \pi(J_\pi) &= &
 \{ 0\}.
 \end{eqnarray*}
It follows from the condition (1) that $\D^*/J_p 
=C^*(M_p ) 
$, since we know that already $C^*(\M)/I_{p_\pi}\simeq C^*(M_{p_\pi}) $. 
Hence it follows that 
\begin{eqnarray*}\label{}
 \nn \pi &\simeq &
 \pi_{p_\pi}\circ  \mu_{p_\pi}
 \end{eqnarray*}
for some $ \pi_{p_\pi}\in\widehat{M_{p_\pi}}$.
 
Therefore the subalgebra $C^*(\M)  $ of $\D^* $ separates the irreducible 
representations of $\D^* $ and so by the
Stone-Weierstrass theorem in \cite{Di}.
\begin{eqnarray*}\label{}
 \nn  \D^*(\M) &=&C^*(\M) .
 \end{eqnarray*}
 \end{proof}

 \end{document}